\documentclass[10pt]{amsart}
\usepackage{palatino}
\usepackage{amsfonts}
\usepackage{amssymb}
\usepackage{amscd}
\usepackage{hyperref}
\usepackage{graphicx}
\usepackage{tikz}

\newtheorem{thm}{Theorem}[section]

\newtheorem{conjecture}[thm]{Conjecture}
\newtheorem{lemma}[thm]{Lemma}
\newtheorem{prop}[thm]{Proposition}
\theoremstyle{definition}
\newtheorem{defn}[thm]{Definition}
\newtheorem{ex}[thm]{Example}
\theoremstyle{rem}
\newtheorem{rem}[thm]{Remark}
\numberwithin{equation}{subsection}
\newtheorem{notation}[thm]{Notation}

\renewcommand{\emptyset}{\varnothing}

\newcommand{\etc}{\,\ldots}

\begin{document}

\title[The tautological ring of the moduli space $M_{2,n}^{rt}$]
{The tautological ring of the moduli space $M_{2,n}^{rt}$}
\author[M. Tavakol]{Mehdi Tavakol}
   \address{Department of Mathematics,
              KTH, 100 44 Stockholm, Sweden}
   \email{tavakol@math.kth.se}

\maketitle
\centerline {\bf Introduction.}
\vskip 1cm

\subsection{Overview}
The moduli space $M_{2,n}^{rt}$ parameterizes the isomorphism classes of stable $n$-pointed curves of genus two with rational tails. This is an open subset of the space $\overline{M}_{2,n}$ and can be considered as a partial compactification of the moduli space $M_{2,n}$, parameterizing smooth $n$-pointed curves of genus two. The degenerated curves parameterized by the moduli space are stable $n$-pointed curves of arithmetic genus two having one component of genus two while the other components are all rational. 
The study of the whole Chow ring $A^*(M_{2,n}^{rt})$ of $M_{2,n}^{rt}$ appears to be a difficult question.
Therefore, we restrict ourselves to the study of the tautological algebra $R^*(M_{2,n}^{rt})$, which is a smaller part of the intersection ring of the moduli space and has an interesting conjectural structure. 
The \emph{tautological classes} on the moduli space of curves are originally defined by Mumford in [M] in his study of the intersection theory of the $Q$-varieties $M_g$ and $\overline{M}_g$. 
The tautological rings are defined to be subalgebras of the rational Chow rings generated by the tautological classes.
Faber has computed the tautological rings in many cases and in [F2] gives a conjectural description of the tautological ring of the moduli spaces of curves. He predicts that the tautological algebras behave like the algebraic cohomology ring of a nonsingular projective variety. In particular, they should be Gorenstein algebras and satisfy the Hard Lefschetz and Hodge Positivity properties with respect to a certain ample class.    

We begin this article by recalling the definition of the moduli space $M_{g,n}^{rt}$ and its tautological ring. We state the related conjectures concerning the structure of $R^*(M_{g,n}^{rt})$. 
Then we focus on the special case $g=2$. For a fixed curve $X$ of genus two and a natural number $n$, there is the notion of the tautological ring $R^*(X^n)$. This ring is defined and studied by Faber and Pandharipande. 
We study this ring and we show that $R^*(X^n)$ is a Gorenstein algebra. As we will see, there are two essential relations in $R^2(X^3)$ and $R^3(X^6)$ which play an important role in our analysis of the tautological rings $R^*(X^n)$. 
The first relation is closely related to the relations discovered by Faber and Belorousski-Pandharipande on different moduli spaces. This gives the result of Theorem \ref{C} which was also proven by Faber-Pandharipande in their unpublished joint work.

To find the connection between $R^*(X^n)$ and the tautological ring of the space $M_{2,n}^{rt}$, we consider the reduced fiber of $\pi:\overline{M}_{2,n} \rightarrow \overline{M}_2$ over the moduli point $[X] \in M_2$, 
which is the Fulton-MacPherson compactification $X[n]$ of the configuration space $F(X,n)$. There is a natural way to define the tautological ring for the space $X[n]$. The study of this algebra shows that it is Gorenstein as well. We conclude this note by showing that there is an isomorphism between the tautological ring of $X[n]$ and $R^*(M_{2,n}^{rt})$. This gives the following result:
\begin{thm}\label{rt}
The tautological ring $R^*(M_{2,n}^{rt})$ is a Gorenstein algebra with socle in degree $n$.
\end{thm}

\subsection{Tautological rings}

Let  $\overline{M}_{g,n}$ be the moduli space of stable curves of genus $g$ with $n$ marked points. In [FP2] the system of tautological rings is defined to 
be the set of smallest $\mathbb{Q}$-subalgebras of the Chow rings, $$R^*(\overline{M}_{g,n}) \subset A^*(\overline{M}_{g,n}),$$satisfying the following two properties:
 
\begin{itemize}

\item
The system is closed under push-forward via all maps forgetting markings: $$\pi_*:R^*(\overline{M}_{g,n}) \rightarrow R^*(\overline{M}_{g,n-1}).$$

\item
The system is closed under push-forward via all gluing maps: 
$$\iota_*: R^*(\overline{M}_{g_1,n_1 \cup \{*\}}) \otimes  R^*(\overline{M}_{g_2,n_2 \cup \{ \bullet \}})  \rightarrow R^*(\overline{M}_{g_1+g_2,n_1+n_2}),$$
$$\iota_*: R^*(\overline{M}_{g,n \cup \{ *,\bullet \}}) \rightarrow R^*(\overline{M}_{g+1,n}),$$ with attachments along the markings $*$ and $\bullet.$
\end{itemize}
Let us recall the definition of the moduli space $M_{g,n}^{rt}$ from [F4]: For $g \geq 2$, it is the inverse image of $M_g$ under the natural morphism $\overline{M}_{g,n} \rightarrow \overline{M}_g$. 
For $g=1$, one has to consider the inverse image of $M_{1,1}$ under the morphism $\overline{M}_{1,n} \rightarrow \overline{M}_{1,1}$, and $M_{0,n}^{rt}=\overline{M}_{0,n}$. 
One can also define $M_{g,n}^{rt}$ as follows:  
To every $n$-pointed stable curve $(C;x_1, \etc,x_n)$ there is an associated dual graph. Its vertices correspond to the irreducible components of $C$ and edges correspond to intersection of components. 
The curve $C$ has rational tails if there is one component whose genus is equal to $g$ and all other components are isomorphic to the projective line $\mathbb{P}^1$. The moduli space $M_{g,n}^{rt}$
parametrizes the isomorphism classes of stable $n$-pointed curves of genus $g$ with rational tails.

The tautological ring, $R^*(M_{g,n}^{rt}) \subset A^*(M_{g,n}^{rt})$, for the moduli space $M_{g,n}^{rt}$, is defined to be the image of $R^*(\overline{M}_{g,n})$ via the natural map, 
$$R^*(\overline{M}_{g,n}) \subset A^*(\overline{M}_{g,n}) \rightarrow A^*(M_{g,n}^{rt}).$$

\subsection{Evaluations}
The quotient ring $R^*(M_{g,n}^{rt})$ admits a canonical non-trivial linear evaluation $\epsilon$ to $\mathbb{Q}$ obtained by integration involving the Chern classes of the Hodge bundle.

Recall that the Hodge bundle $\mathbb{E}$ on $\overline{M}_g$ for $g>1$ (resp. $\overline{M}_{1,1}$ for $g=1$), is the locally free $Q$-sheaf of rank $g$  defined by  $\mathbb{E}=\pi_* \omega$, 
where $\pi:\overline{M}_{g,1} \rightarrow \overline{M}_g$ (resp. $\pi:\overline{M}_{1,2} \rightarrow \overline{M}_{1,1}$) is the universal curve of genus $g$ and $\omega$ denotes its relative dualizing sheaf. 
The Hodge bundle on $\overline{M}_{g,n}$ is defined as the pull-back of $\mathbb{E}$ via the natural projection 
$\pi:\overline{M}_{g,n} \rightarrow \overline{M}_g$ for $g>1$ 
(resp. $\pi:\overline{M}_{1,n} \rightarrow \overline{M}_{1,1}$ for $g=1$)
and is denoted by the same letter.
The fiber of $\mathbb{E}$ over a moduli point $[(C;x_1, \dots,x_n)]$ is the $g$-dimensional vector space $H^0(C,\omega_C)$. 
The class $\lambda_i$ on $\overline{M}_{g,n}$ is defined to be the $i^{th}$ Chern class $c_i(\mathbb{E})$ of the Hodge bundle.  
The class $\psi_i$ is the pull back $\sigma_i^*(K)$ of $K$ along $\sigma_i:\overline{M}_{g,n} \rightarrow \overline{M}_{g,n+1}$, 
where $\sigma_1, \etc,\sigma_n$ are the natural sections of the map $\pi:\overline{M}_{g,n+1} \rightarrow \overline{M}_{g,n}$, which forgets the last marking on the curve and stabilizes. It is the first Chern class of the cotangent line bundle $\mathbb{L}_i$ on the moduli space whose fiber at the moduli point $[(C;x_1, \etc,x_n)]$ is the cotangent space to $C$ at the $i^{th}$ marking.
The class $\kappa_i$ on $\overline{M}_{g,n}$ is defined to be the push-forward $\pi_*(\psi_{n+1}^{i+1})$, where the projection
$\pi:\overline{M}_{g,n+1} \rightarrow \overline{M}_{g,n}$ is defined as above.

It is proven in [F2] that for $g>0$ the class $\lambda_{g-1} \lambda_g$ vanishes when restricted to the complement of the open subset $M_{g,n}^{rt}$. This leads to an evaluation $\epsilon$ on $A^*(M_{g,n}^{rt})$:
$$\xi \mapsto \epsilon(\xi)=\int_{\overline{M}_{g,n}} \xi \cdot \lambda_{g-1} \lambda_{g}.$$

There is an interesting connection between the non-triviality of the pairing above and the tautological ring of $M_g$: In [F1] Faber studies the algebra $R^*(M_g)$
and conjectures that it is a Gorenstein ring with socle in degree $g-2$. The Hodge integrals 
\begin{equation} \tag{1}\label{Int} \int_{\overline{M}_{g,n}} \psi_1^{\alpha_1} \dots \psi_n^{\alpha_n} \lambda_{g-1} \lambda_g \end{equation}
determine the top intersection pairings in $R^*(M_g)$. The study of the tautological ring of $M_g$ in [F1] led Faber to the following conjecture for the integrals \eqref{Int}:
$$\int_{\overline{M}_{g,n}} \psi_1^{\alpha_1} \dots \psi_n^{\alpha_n} \lambda_{g-1} \lambda_{g}= \frac{(2g+n-3)! (2g-1)!!}{(2g-1)! \prod_{i=1}^n (2 \alpha_i -1)!!} \int_{\overline{M}_{g,1}} \psi_1^{g-1} \lambda_{g-1} \lambda_{g},$$
where $g \geq 2$ and $\alpha_i \geq 1$. In [GP] it is shown that the degree zero Virasoro conjecture applied to $\mathbb{P}^2$ implies this prediction. The constant 
$$\int_{\overline{M}_{g,1}} \psi_1^{g-1} \lambda_{g-1} \lambda_{g}=\frac{1}{2^{2g-1}(2g-1)!!}\frac{|B_{2g}|}{2g}$$
has been calculated by Faber, who shows that it follows from Witten's conjecture. The non-vanishing of the Bernoulli number $B_{2g}$ shows the non-triviality of the evaluation $\epsilon$. For more details about the role of Hodge integrals in the study of the tautological rings and its connection with Virasoro constraints in Gromov-Witten theory, see [FP1] or [F3].

\subsection{Gorenstein Conjectures}
Recall the following conjectures by Faber-Pandharpiande:
\begin{conjecture} \label{mg} $R^*(M_{g,n}^{rt})$ is Gorenstein with socle in degree $g-2+n-\delta_{0g}.$ \end{conjecture}
\begin{conjecture} \label{bar} $R^*(\overline{M}_{g,n})$ is Gorenstein with socle in degree $3g-3+n$. \end{conjecture}

Hain and Looijenga introduce a compactly supported version of the tautological algebra: The algebra $R^*_c(M_{g,n})$ is defined to be the set of elements in 
$R^*(\overline{M}_{g,n})$ that restrict trivially to the Deligne-Mumford boundary. This is a graded ideal in $R^*(\overline{M}_{g,n})$ and the intersection product defines a map 
$$R^*(M_{g,n}) \times R^{*}_c(M_{g,n}) \rightarrow R^{*}_c(M_{g,n})$$ that makes $R^*_c(M_{g,n})$ a $R^*(M_{g,n})$-module. 
In [HL] they formulated the following conjecture for the case $n=0$:
\begin{conjecture}\label{hl}
(A) The intersection pairings $$R^k(M_g) \times R^{3g-3-k}_c(M_g) \rightarrow R_c^{3g-3}(M_g) \cong \mathbb{Q}$$ are perfect for $k \geq 0.$

(B) In addition to (A), $R^*_c(M_g)$ is a free $R^*(M_g)$-module of rank one. 
\end{conjecture}
There is a generalization of the notion of the compactly supported tautological algebra to the space $M_{g,n}^{rt}$: In [F4] Faber defines $R^*_c(M_{g,n}^{rt})$ as the set of elements in $R^*(\overline{M}_{g,n})$
that restrict trivially to $\overline{M}_{g,n} \backslash M_{g,n}^{rt}$. He considers the following generalization of the conjectures above:
\begin{conjecture} \label{HL}
(C) The intersection pairings 
$$R^k(M_{g,n}^{rt}) \times R^{3g-3+n-k}_c(M_{g,n}^{rt}) \rightarrow R_c^{3g-3+n}(M_{g,n}^{rt}) \cong \mathbb{Q}$$ are perfect for $k \geq 0.$

(D) In addition to C, $R^*_c(M_{g,n}^{rt})$ is a free $R^*(M_{g,n}^{rt})$-module of rank one. 
\end{conjecture}
The study of the intersection ring of $\overline{M}_{0,n}$ by Keel [K] gives a proof of the conjectures above in the case $g=0$. 
In [T2] we give another proof of this fact by giving a different construction of the moduli space $\overline{M}_{0,n}$,
which leads to an explicit duality between the Chow groups in complementary degrees. Conjecture \ref{mg} in the case $n=0$ has been verified by Faber for $g \leq 23$. 
In [F4] Faber shows that Conjectures \ref{mg} and \ref{bar} are true for all $(g,n)$ if and only if Conjecture (D) is true for all $(g,n)$. 
In our work [T1] we proved that Conjecture \ref{mg} is true for $g=1$.
In this paper we prove this conjecture for the case $g=2$.

\vspace{+10pt}
\noindent{\bf Acknowledgments.}
I  am thankful to my advisor Carel Faber for useful discussions and comments. 
He informed me about the unpublished joint work with Rahul Pandharpipande on the tautological ring of curves. 

\section{The tautological ring  $R^*(X^{n})$}

\begin{defn}
Let $X$ be a smooth curve of genus two and $K$ be its canonical divisor. For a natural number $n \in \mathbb{N}$, the tautological ring,  $R^*(X^n) \subset A^*(X^n)$, is defined to be the 
$\mathbb{Q}$-subalgebra generated by the divisor classes $K_i,d_{j,k}$, for $1 \leq i \leq n$ and $1 \leq j < k \leq n$. Here, the class $K_i$ is the pull-back $\pi_i^*(K)$, where $\pi_i:X^n \rightarrow X$ denotes the projection onto the $i^{th}$ factor, and 
$d_{j,k}$ stands for the class of the big diagonal defined by the equation $x_j=x_k$. 

To describe the tautological algebra it is more convenient to introduce other classes in terms of the tautological classes. 
Let $a \in X$ be a Weierstrass point on the curve. From the rational equivalence $K=2a$ one gets the equality $K_i=2a_i$, where $a_i=\pi_i^*(a)$.
If we define $b_{j,k}:=d_{j,k}-a_j-a_k$, then another set of generators for $R^*(X^n)$ is $$\{a_i,b_{j,k}: \ 1 \leq i \leq n \ \mathrm{and} \ \ 1 \leq j < k \leq n\}.$$
\end{defn}

\begin{rem}\label{tc}
The generators of the tautological ring $R^*(X^n)$ are pull-backs of divisor classes defined on certain moduli spaces of curves: Let $C_g=M_{g,1}$ be the universal curve of genus $g$ and $C_g^n$ be the n-fold fiber product of $C_g$ over $M_g$, parameterizing smooth curves of genus $g$ with $n$-tuples of not necessarily distinct points.  
Denote by $\pi:C_g \rightarrow M_g$ the morphism forgetting the point and let $\omega$ be its relative dualizing sheaf.  The class of $\omega$ in $A^1(C_g)$ is denoted by $K$.
In [L] the tautological ring, $R^*(C_g^n) \subset A^*(C_g^n)$, is defined to be the subalgebra generated by the divisor classes $K_i=c_1(\omega_i)$,
where $\omega_i=\pi_i^*(\omega)$ and $\pi_i:C^n_g \rightarrow C_g$ is the projection onto the $i^{th}$ factor, $D_{j,k}$ (the class of the diagonal $x_j=x_k$) and the pull-backs from $M_g$ of the kappa classes $\kappa_i$. The curve $X$ above defines a canonical map from $X^n$ to $C_2^n.$ The pull-back of the classes $K_i$ and $D_{i,j}$ to $X^n$ via 
this map gives the classes defined above. Recall that $\kappa_i=0$ on $M_2$ for $i \geq 1$ since Mumford proves in [M] that theses kappa classes live on the boundary of the Deligne-Mumford compactification $\overline{M}_2$ of $M_2$.
\end{rem}
The main result of this section is the following theorem:
\begin{thm}\label{C}
(A) The space of relations in $R^*(X^n)$ is generated by the following ones: $$a_i^2=0, \qquad a_ib_{i,j}=0, \qquad  b_{i,j}^2=-4a_ia_j, \qquad b_{i,j}b_{i,k}=a_ib_{j,k}, \qquad \sum b_{i_1,i_2}b_{i_3,i_4}b_{i_5,i_6}=0,$$
where in each relation the indices are distinct. For every subset $\{i_1, \etc, i_6\}$ of $\{1, \etc, n\}$ with 6 elements,
there is a relation of the final type, where the sum is taken over all partitions of the set $\{i_1, \etc , i_6\}$ 
into 3 subsets with 2 elements.   

(B) For any $0 \leq d \leq n$, the pairing $R^d(X^n) \times R^{n-d}(X^n) \rightarrow \mathbb{Q}$ is perfect.
\end{thm}
\begin{proof}
We first verify that the relations above hold in the tautological ring $R^*(X^n)$. The relations $a_i^2=a_ib_{i,j}=0$ and $b_{i,j}^2=-4a_ia_j$ obviously hold. 
The relation $b_{i,j}b_{i,k}=a_ib_{j,k}$ follows from a known relation in $R^2(C_2^3)$: In [F2] Faber shows that the following relation holds in $R^2(C^3_2)$:
\begin{equation} \tag{2}\label{F} K_1D_{1,2}+K_1D_{1,3}+K_2D_{2,3}-K_1D_{2,3}-K_2D_{1,3}-K_3D_{1,2}+2D_{1,2}D_{1,3}=0. \end{equation}
The pull-back of the relation above via the morphism $X^3 \rightarrow C_2^3$ gives $2(b_{1,2}b_{1,3}-a_1b_{2,3})=0$.

There is another way to see this: In [BP] a genus 2 relation among codimension 2 descendent stratum classes in $\overline{M}_{2,3}$ is found.  
The statement is given in Theorem 1 of [BP]. We don't rewrite it here since it is a rather long expression. Instead, we explain how it gives the desired relation in $R^2(X^3)$: 
If we restrict that relation to $M_{2,3}^{rt}$ and pull it back to $R^2(X[3])$ via the morphism $F:X[3] \rightarrow M_{2,3}^{rt}$, defined in the second section, we get a relation whose push-forward via the proper
morphism $X[3] \rightarrow X^3$ gives 12 times the relation $a_1b_{2,3}-b_{1,2}b_{1,3}=0$.

We use a method similar to the one in [F2] to get the last class of relations.
Let $\pi:C_2 \rightarrow M_2$ be the universal curve over $M_2$ and denote by $\omega$ its relative dualizing sheaf.  Consider $\mathbb{E}_3=\pi_*(\omega^{\otimes 3})$. It is a vector bundle of rank 5; its fiber at the point $C$ is the vector space 
$H^0(C,3K)$. Consider the space $C_2^d$, the $d$-fold fiber product of $C_2$ over $M_2$, and denote by $\pi:C^{d+1}_2 \rightarrow C_2^d$ the morphism forgetting the $(d+1)$-st point. Let $\Delta_{d+1}$ be the sum of the $d$ divisors $D_{1,d+1}, \etc , D_{d,d+1}$: $$\Delta_{d+1}=D_{1,d+1}+ \etc +D_{d,d+1}.$$
Let $\mathbb{F}_d=\pi_*(\mathcal{O}_{\Delta_{d+1}}\otimes \omega_{d+1} ^{\otimes 3})$.
The sheaf $\mathbb{F}_d$ is locally free of rank $d$; its fiber at a point $(C;x_1, \dots ,x_d)=(C;D)$ is the vector space 
$H^0(C,3K/3K(-D))$.
It is easy to see that the total Chern class of $\mathbb{F}_d$ can be expressed in terms of the tautological classes on $C_2^d$: $$c(\mathbb{F}_d)=(1+3K_1)(1+3K_2-\Delta_2) \etc (1+3K_d-\Delta_d).$$
There is a natural evaluation map $\phi_d:\mathbb{E}_3 \rightarrow \mathbb{F}_d$ between the pull-back of $\mathbb{E}_3$ and the bundle $\mathbb{F}_d$. Fiberwise the kernel over $(C;D)$ is the vector space $H^0(C,3K(-D))$. 
When $d=7$ the map $\phi_d$ is an injection, from which one gets the relation $c_3(\mathbb{F}_7-\mathbb{E}_3)=0$. 
Then we multiply this expression by $a_7$. 
The push-forward of the resulting expression via the morphism $\pi:C_2^7 \rightarrow C_2^6$,
which forgets the last coordinate, 
yields a relation whose pull-back to $R^3(X^6)$ via the morphism $X^6 \rightarrow C_2^6$ is equal to minus the desired relation: 
\begin{equation} \tag{3}\label{6} \sum b_{i_1,i_2}b_{i_3,i_4}b_{i_5,i_6}=0. \end{equation}

Using the proven relations in the tautological algebra we are able to study the pairing $R^d(X^n) \times R^{n-d}(X^n) \rightarrow \mathbb{Q}$, for $0 \leq d \leq n$. 
From the relations above, we see that the tautological group $R^d(X^n)$ is generated by monomials of the form 
$v=a(v)b(v)$, where $a(v)$ is a product $\prod a_i$ of $a_i$'s for $i \in A_v$, and $b(v)$ is a product $\prod b_{j,k}$ of $b_{j,k}$'s, for $j,k \in B_v$, 
such that $A_v$ and $B_v$ are disjoint subsets of the set $\{1, \etc ,n\}$ satisfying $d=|A_v|+\frac{1}{2}|B_v|$. Under this circumstance, 
the monomial $v$ is said to be \emph{standard}. To any standard monomial $v$ we associate a dual element $v^* \in R^{n-d}(X^n)$, 
which is defined to be the product of all $a_i$'s, for $i \in \{1, \etc ,n\}-A_v \cup B_v$, with $b(v)$. The following lemma enables us to study the pairing:
\begin{lemma}
Let $v\in R^d(X^n)$ and $w \in R^{n-d}(X^n)$ be standard monomials. If the product $v \cdot w$ is nonzero, then $B_v=B_w$, 
and the disjoint union of the sets $A_v,A_w$ and $B_v=B_w$ is equal to the set $\{1, \etc ,n\}$.
\end{lemma}
\begin{proof}
The proof is easy and can be found in [T1].
\end{proof}
So, after a suitable enumeration of generators for $R^d(X^n)$, the resulting intersection matrix of the pairing between standard 
monomials and their duals consists of square blocks along the main diagonal and the off-diagonal blocks are all zero. To prove that the pairing is perfect we need to study the square blocks on the main diagonal. 
These matrices and their eigenvalues are studied in [HW]. In particular, from their result it follows that the kernel of any such matrix is generated by relations of the form \eqref{6}:

\begin{lemma}
Let $m \geq 2$ be a natural number and $S$ be the set of all standard monomials $v$ of the form $b_{i_1,j_1} \dots b_{i_m,j_m}$ in $R^m(X^{2m})$. The kernel of the intersection matrix $(v \cdot w)$ for $v,w$ in $S$
is generated by expressions of the form \eqref{6}.
\end{lemma}
\begin{proof}
The proof is identical to that of Lemma 4.4. in [T1]. We briefly recall the proof:
The intersection matrix $(v \cdot w)$ for $v,w \in S$ in [HW] is denoted by $T_r(x)$ for $r=m$ and $x=-4$. The $S_{2m}$-module generated by elements of $S$ decomposes into the sum $\oplus_{\lambda} V_{\lambda}$, 
where $\lambda$ varies over all partitions of $2m$ into even parts. For each such $\lambda$ the space $V_{\lambda}$ is an eigenspace of $T_r(x)$. 
The corresponding eigenvalue is zero  if an only if the partition $\lambda$ has a part of size at least six. 
For each such $\lambda$ it is easy to identify the space $V_{\lambda}$ with a subspace of $R^m(X^{2m})$ generated by expressions 
of the form \eqref{6}. This proves the statement.
\end{proof}

Since the relations of the form \eqref{6}  hold in the tautological ring $R^*(X^n)$, we conclude that the pairing is perfect. 
This also shows that the relations stated in the theorem generate all relations in the tautological ring.
\end{proof}

\section{The description of the fiber $X[n]$}
In the previous section we considered a fixed smooth curve $X$ of genus two and we studied the tautological ring $R^*(X^n)$ for a natural number $n \in \mathbb{N}$. In this section we look at the reduced fiber of 
the projection $\pi: \overline{M}_{2,n} \rightarrow \overline{M}_2$, which forgets the markings on the curve and stabilizes, over the moduli point $[X] \in M_2$. This fiber is the Fulton-MacPherson compactification $X[n]$ of the configuration space $F(X,n)$. 
Let us recall from [FM] the related definitions and some basic properties of this space: The configuration space $$F(X,n)=X^n \backslash \bigcup \Delta_{\{a,b\}}$$
is the complement in the cartesian product of the large diagonals $\Delta_{\{a,b\}}$ where the points with two labels $a$ and $b$ coincide. The compactification $X[n]$ is defined as follows: For each subset $S$ of $\{1, \etc ,n\}$
with at least two points let $\mathrm{Bl}_{\Delta}(X^S)$ denote the blow-up of the corresponding cartesian product $X^S$ along its small diagonal. There is a natural embedding $$F(X,n) \subset X^n \times \prod_{|S| \geq 2} \mathrm{Bl}_{\Delta}(X^S),$$
and $X[n]$ is defined to be the closure of the configuration space in this product. In Section 3 of [FM] an explicit construction of $X[n]$ by a sequence of blow-ups is given. Their construction is inductive: They define $X[1]$ to be $X$. 
Assuming that the space $X[n]$ is already constructed, there is a sequence of blow-ups of the product $X[n] \times X$, corresponding to a class of diagonals, which gives $X[n+1]$. 

It is shown in [FM] that
\begin{itemize}
\item $X[n]$ is an irreducible variety;
\item the canonical map from $X[n]$ to $X^n$ is proper;
\item $X[n]$ is symmetric,
\end{itemize} 
i.e., the symmetric group $\mathrm{S}_n$ acts on $X[n]$. Another important property of the variety $X[n]$ is that it is nonsingular. For two points this space coincides with the product $X \times X$. The space $X[3]$ is obtained from the product $X^3$ by blowing up the small diagonal. As it is mentioned in [FM], "for $n > 3$, however, if one starts by blowing up the small diagonal in $X^n$ and then blows up proper transforms of the next larger diagonals, then the proper transform of succeeding diagonals will not be separated, so extra blow-ups are needed to get a smooth compactification. It would be interesting to see if other sequences of blow-ups give compactifications that are symmetric, and whose points have explicit and concise descriptions." 

In [Li] the notion of the wonderful compactification of an arrangement of subvarieties is defined. It is proven that for a nonsingular algebraic variety $X$, the space $X[n]$ is a wonderful compactification corresponding to a certain arrangement of subvarieites. 
This leads to a construction of $X[n]$ as a symmetric sequence of blow-ups in the order of ascending dimension.   

We will describe it as follows: Let $I \subset \{1, \etc ,n\}$ be a subset with $|I| \geq 3$. The subvariety $X_I$ of $X^n$ is defined to be the set of all points $(x_1, \etc ,x_n)$ in $X^n$ where the coordinates corresponding to the index set $I$ are equal to each other. The space $X[n]$ is obtained from $X^n$ by the following sequence of blow-ups: First blow up the small diagonal corresponding to the set $I=\{1, \etc , n\}$. At each stage we increase the dimension of the blow-up center by one. At the $k^{th}$ step, for $1 \leq k \leq n-2$, the blow-up is along the union of the proper transforms of the subvarieties $X_I$ for which $|I|=n+1-k$. This can be done by blowing-up these subvarieties in any order.
The exceptional divisor of the blow-up along the subvariety $X_I$ is denoted by $D_I$ as well as the class of its proper transform under later blow-ups.  

The connection between the space $X[n]$ and the moduli space $M_{2,n}^{rt}$ is explained as follows: Let $\pi:X^{n+1}=X^n \times X \rightarrow X^n$ be the projection onto the first factor. This is the trivial family of curves of genus two parameterized by the points of $X^n$ with $n$ disjoint sections $\sigma_1, \etc ,\sigma_n$ over the open subset $F(X,n)$. This data induces the morphism $F:F(X,n) \rightarrow M_{2,n}$, which is explicitly given by the rule $$(x_1, \etc ,x_n) \rightarrow [(X;x_1, \etc ,x_n)].$$
The Fulton-MacPherson compactification $X[n]$ of the configuration space $F(X,n)$ gives rise to an extension of the morphism $F$ to the moduli space $M_{2,n}^{rt}$, which will be denoted by the same letter, by abuse of notation. More precisely, the projection $\pi: X[n+1] \rightarrow X[n]$ is a family of stable curves of genus two with rational tails. The morphism $\pi$ admits $n$ disjoint sections in the smooth locus of its fibers. These sections, which restrict to previously defined ones, will be denoted by the same letters, another abuse of notation. The morphism $F:X[n] \rightarrow M_{2,n}^{rt}$ is defined by the rule $$P \rightarrow [(\pi^{-1}(P); \sigma_1(P), \etc ,\sigma_n(P))].$$ 
This morphism will be our main tool in the comparison between the tautological ring of the moduli space $M_{2,n}^{rt}$ and the algebra $R^*(X^n)$ defined in the previous section. 

\section{The tautological ring  $R^*(X[n])$}
In the first part we defined the tautological ring of the product $X^n$ for a fixed curve $X$ of genus two and we proved that it is a Gorenstein algebra. 
In this section we define the tautological algebra $R^*(X[n])$ of the fiber $X[n]$ and we prove that it is Gorenstein as well. We first recall some general facts from [FM] about the intersection ring of the blow-up of a
smooth variety $Y$ along a smooth irreducible subvariety $Z$. When the restriction map from $A^*(Y)$ to $A^*(Z)$ is surjective, S. Keel has shown in [K] that the computations become simpler. 
We denote the kernel of the restriction map by $J_{Z/Y}$ so that $$A^*(Z)=\frac{A^*(Y)}{J_{Z/Y}}.$$ Define a Chern polynomial for $Z \subset Y$, denoted by 
$P_{Z/Y}(t)$, to be a polynomial $$P_{Z/Y}(t)=t^d+a_1t^{d-1}+ \dots +a_{d-1}t+a_d \in A^*(Y)[t],$$ where $d$ is the codimension of $Z$ in $Y$ 
and $a_i \in A^i(Y)$ is a class whose restriction in $A^i(Z)$ is $c_i(N_{Z/Y})$, where $N_{Z/Y}$ is the normal bundle of $Z$ in $Y$. We also require that $a_d=[Z]$, 
while the other classes $a_i$, for $0<i<d$, are determined only modulo $J_{Z/Y}$.
We identify $A^*(Y)$ as a subring of $A^*(\widetilde{Y})$ by means of the map $\pi^*:A^*(Y) \rightarrow A^*(\widetilde{Y})$, 
where $\pi:\widetilde{Y}\rightarrow Y$ is the birational morphism. Let $E \subset \widetilde{Y}$ be the exceptional divisor. The formula of Keel is as follows:
The Chow ring $A^*(\widetilde{Y})$ is given by $$A^*(\widetilde{Y})=\frac{A^*(Y)[E]}{(J_{Z/Y} \cdot E,P_{Z/Y}(-E))}.$$

\begin{rem}
As we saw in the second section, the fiber $X[n]$ is described as a sequence of blow-ups of the variety $X^n$. 
It is easy to see that the restriction map from $A^*(Y)$ to the Chow ring of a blow-up center $Z$ in the construction of the space $X[n]$ is surjective. This means that we can apply the formula of Keel. 
Therefore, there are two essential data needed in the computation of the intersection ring at each step: the ideal $J_{Z/Y}$ and a Chern polynomial $P_{Z/Y}$. These are easy to compute when $Y=X^n$ 
and $Z=X_I$, for a subset $I \subset \{1, \dots ,n\}$. The relation between the ideals $J_{V/Y}$ and $J_{\widetilde{V}/\widetilde{Y}}$, where $\widetilde{Y}=\mathrm{Bl}_Z Y$ and $V \subset Y$ is a subvariety, is discussed in [FM]
as well as the formula relating a Chern polynomial $P_{V/Y}$ to $P_{\widetilde{V}/\widetilde{Y}}$.
\end{rem}

\begin{defn}\label{T}
The tautological ring, $R^*(X[n])$ of $X[n]$, is defined to be the subalgebra of the Chow ring $A^*(X[n])$ generated by the pull-backs of the tautological classes in $R^*(X^n)$ and the classes of the exceptional divisors $D_I$. 
\end{defn}

\begin{rem}
As we observed in Remark \ref{tc}, the generators of the tautological algebra $R^*(X^n)$ are pull-backs of natural classes on certain moduli spaces of curves. There is a similar situation for the space $X[n]$: 
The tautological classes on the fiber $X[n]$ are simply the restrictions of the generators of $R^*(M_{2,n}^{rt})$ to $X[n]$ via the map $F:X[n] \rightarrow M_{2,n}^{rt}$.
\end{rem}

From the analysis of the intersection ring of $X[n]$ we obtain the following description of the relations:
\subsection{Relations in $R^*(X[n])$}
\begin{itemize}\label{R}
\item
The first class of relations in $R^*(X[n])$ are those that hold among the generators of $R^*(X^n)$ described in the second section. Notice that the tautological algebra of $X^n$ is identified as a subalgebra of $R^*(X[n])$.

\item For subsets $I,J \subset \{1, \etc ,n\}$, where $|I|,|J| \geq 3$, the product $D_I \cdot D_J \in R^2(X[n])$ is zero unless $$(*) \qquad I \subseteq J, \qquad  \mathrm{or} \ J \subseteq I, \qquad \mathrm{or} \ I \cap J=\emptyset.$$

\item For a subset $I \subset \{1, \etc ,n\}$ with $|I| \geq 3$, consider the inclusion morphism $i_I:X_I \rightarrow X^n$. The relation $x \cdot D_I=0$ holds for $x \in \ker(i^*_I: R^*(X^n) \rightarrow R^*(X_I))$. This gives a set of relations which 
coincides with the kernel of the operator $X_I \cap -:R^*(X^n) \rightarrow R^*(X^n)$ defined by the rule $x \rightarrow X_I \cap x$.

\item Let $Z$ be a blow-up center and assume that it is written as the transversal intersection $V \cap W$ of the subvarieties $V$ and $W$ of the variety $Y$.
If $V$ is the transversal intersection $V_1 \cap \dots \cap V_k$ and the proper transforms $\widetilde{V}_1, \dots ,\widetilde{V}_k$ of the subvarieties $V_1, \dots ,V_k$ are centers of later blow-ups, 
then from Lemma 5.5 in [FM], we get the relation $P_{W/Y}(-E_Z) \cdot E_{\widetilde{V}_1} \dots  E_{\widetilde{V}_k}=0$. 

\item If $Z$ is a blow-up center of the variety $Y$ with a Chern polynomial $P_{Z/Y}$, there is a relation $P_{Z/Y}(-E_Z)=0$. This gives the last class of relations in the tautological ring of $X[n]$.
\end{itemize}

\subsection{Standard monomials}
The tautological ring of the space $X[n]$ is generated over $R^*(X^n)$ by the classes of the exceptional divisors $D_I$. The generators for the ideal of relations are described in \ref{R}.
These relations are used to obtain a smaller set of generators for the ring $R^*(X[n])$. This leads us to the notion of 
\emph{standard monomials}. 
To enumerate the monomials we associate a graph to the generator $v$. We first define an ordering on the polynomial ring 
$$R:=\mathbb{Q}[a_i,b_{j,k},D_I: 1 \leq i \leq n,1 \leq j<k \leq n,I \subset \{1,\etc,n\}, \ \mathrm{where} \ |I|\geq 3].$$
\begin{defn}\label{<}
Let $I,J \subset \{1, \etc ,n\}$, we say that $I<J$ if
\begin{itemize}
\item $|I|<|J|$

\item or if $|I|=|J|$ and the smallest element in $I - I \cap J$ is smaller than the smallest element of $J-I\cap J$.
\end{itemize}
Put an arbitrary total order on monomials in $$\mathbb{Q}[a_i,b_{j,k}:1 \leq i \leq n, 1 \leq j < k \leq n].$$
Suppose $v_1,v_2 \in R$ are monomials. We say that $v_1 < v_2$ if we can write them as 
$$v_1=a(v_1)b(v_1) \cdot \prod_{r=1}^{r_0}D_{I_r}^{i_r} \cdot D, \qquad \  v_2=a(v_2)b(v_2) \cdot \prod_{r=1}^{r_0}D_{I_r}^{j_r} \cdot D, \qquad \mathrm{where} \ D=\prod_{r=r_0+1}^m D_{I_r}^{i_r},$$ 
for $I_m < \etc <I_1$, and $i_{r_0}<j_{r_0}$; 
or if $r_0=0$ and $a(v_1)b(v_1)<a(v_2)b(v_2)$. 
Furthermore, we say that  $v_1 \ll v_2$, if for any factor $D_I$ of $v_2$ we have that $v_1 <D_I$. Note that $v_1 \ll v_2$ implies that $v_1 <v_2$. 
\end{defn}

\begin{defn}\label{Graph}
Let $v=a(v)\cdot D_{I_1}^{i_1} \etc D_{I_m}^{i_m}$, where  $i_r \neq 0$ for  $r=1, \etc ,m$ and $I_m< \etc < I_1$, be a monomial. 
The directed graph $\mathcal{G}=(V_\mathcal{G},E_\mathcal{G})$ associated to $v$ or of the collection $I_1, \etc ,I_m$ is defined by the following data:
\begin{itemize}
\item A set $V_\mathcal{G}$ and a one-to-one correspondence between members of $V_{\mathcal{G}}$ and members of the set $\{1, \etc ,m\}$. The elements of $V_\mathcal{G}$ are called the vertices of $\mathcal{G}.$
\item  A set $E_\mathcal{G} \subset V_\mathcal{G} \times V_\mathcal{G}$ consisting of all pairs $(r,s)$, where $I_s$ is a maximal element of the set $$\{I_i: I_i \subset I_r\}$$ 
with respect to inclusion. The elements of $E_\mathcal{G}$ are called the edges of $\mathcal{G}.$
\end{itemize}
For a vertex $i \in V_\mathcal{G}$, the closure $\overline{i} \subset V_\mathcal{G}$ is defined to be the subset
$$\{r \in V_\mathcal{G}: I_r \subseteq I_i\}$$ of $V_\mathcal{G}.$ The degree $\deg(i)$ of $i$ is defined to be the number of the elements of the set 
$$\{j \in V_\mathcal{G}: (i,j) \in E_\mathcal{G}\}.$$

A vertex $i \in V_\mathcal{G}$ is called a \emph{root} of $\mathcal{G}$ if $I_i$ is maximal with respect to inclusion of sets. 
It is called \emph{external} if $I_i$ is a minimal subset and all the other vertices will be called \emph{internal}.

In the following, we use the letters $I_1,\etc,I_m$ to denote the vertices of $\mathcal{G}$.
\end{defn}

\begin{rem}
It is easy to see that the graph $\mathcal{G}$ associated to a non-zero monomial $v$ has no loop. Therefore, we refer to $\mathcal{G}$ as the associated forest of $v$ or of the collection $I_1,\etc,I_m$.
\end{rem}

\begin{ex}\label{ex}
Let $n=20$ and $v=\prod_{i=r}^7 D_{I_r}$, where the subsets $I_1, \etc , I_7$ of the set $\{1, \etc , 20\}$ are defined as follows: 
$$I_1=\{1, \etc , 8\}, \qquad I_2=\{1,2,3\}, \qquad I_3=\{4,5,6\},$$
$$I_4=\{9,\etc,20\}, \qquad I_5=\{9, \etc , 18\}, \qquad I_6=\{9,10,11,12\}, \qquad  I_7=\{13,14,15,16\}.$$
The graph $\mathcal{G}$ associated to the monomial $v$ is pictured below. 
The graph $\mathcal{G}$ has 7 vertices and 5 edges. The vertices $I_1,I_5$ have degree two, the vertex $I_4$ is of degree one and all the other vertices are of degree zero. The vertices $I_1,I_4$ are roots of the graph. External vertices are $I_2,I_3,I_6,I_7$
and the vertices $I_1,I_4,I_5$ are internal.  

\begin{figure}[htp]
\begin{tikzpicture}
[scale=.8,auto=left,every node/.style={circle,fill=blue!20}]
\node (n1) at (2,8) {$I_1$};
\node (n2) at (1,10)  {$I_2$};
\node (n3) at (3,10) {$I_3$};
\draw[>=latex,->] (n1) to (n2);
\draw[>=latex,->] (n1) to (n3);
\node (n4) at (6,8) {$I_4$};
\node (n5) at (6,10)  {$I_5$};
\node (n6) at (5,12) {$I_6$};
\node (n7) at (7,12) {$I_7$};
\draw[>=latex,->] (n4) to (n5);
\draw[>=latex,->] (n5) to (n6);
\draw[>=latex,->] (n5) to (n7);
\end{tikzpicture}
\caption{The graph $\mathcal{G}$}
\end{figure}
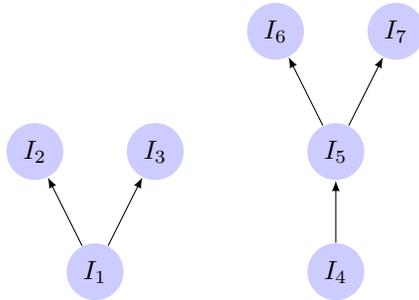
\end{ex}

\begin{defn}\label{standard}
Let $v$ be as in Definition \ref{Graph}, $\mathcal{G}$ be the associated forest, and $J_1,\etc,J_s$, for some $s \leq m$, be the roots of $\mathcal{G}.$ For each $1 \leq r \leq s$ let $\alpha_r \in J_r$ 
be the smallest element. The subset $S$ of the set $\{1,\etc,n\}$ is defined as follows:
$$S:=\{\alpha_1,\etc,\alpha_s\} \cup (\cap_{r=1}^m I_r^c),$$

The monomial $v$ is said to be {\bf standard} if
\begin{itemize}
\item The monomial $a(v)b(v)\in R^*(X^S)$ is in standard form according to the definition given in the first section.
\item For each $r$ we have that $$i_r  \leq \min(|I_r|-2, |I_r|-|\cup_{I_s \subset I_r} I_s|+\deg(I_r)-2).$$
\end{itemize}
\end{defn}

\begin{ex} Let $I \subset J$ be subsets of the set $\{1, \etc , n\}$ with at least 3 elements. The monomial $D_I^i$ is standard if and only if $i <|I|-1$. The monomial $D_J^ j D_I^i$ is standard if and only if $j < |J|-|I|$ and $i < |I|-1$. 
\end{ex}

\begin{rem} \label{st}
From the relations described in \ref{R}, it is easy to see that the tautological group $R^d(X[n])$, for $0 \leq d \leq n$, is additively generated by standard monomials. 
\end{rem}

\subsection{Definition of the dual element}
One advantage of introducing the standard monomials is that they give a smaller set of generators for the tautological groups. This is useful from a computational point of view.
Another important conclusion is that there is a natural involution on the tautological ring which gives a one to one correspondence between standard monomials in complementary degrees.

\begin{defn}\label{dual}
Suppose $v=a(v)b(v) \cdot D(v)$ is a standard monomial, where $a(v)b(v)$ is in the tautological ring $R^*(X^n)$ of $X^n$, and $$D(v)=\prod_{r=1}^m D_{I_r}^{i_r},$$ 
where $i_r \neq 0$ for $r=1,...,m$, and $I_m < ...<I_1$. Let $\mathcal{G}$ be the associated forest, and $J_1, \etc,J_s$, the set $S$ and elements $\alpha_r \in J_r$ for $1 \leq r \leq s$ be as in Definition \ref{standard}. 
The subset $T$ of the set $\{1,\etc,n\}$ is defined to be $$T:=S-A_v \cup B_v.$$
For each $1 \leq r \leq m$, define $j_r$ to be 
$$ \left\{ \begin{array}{ll}
|I_r|-|\cup_{I_s \subset I_r}I_s|+\deg(I_r)-1-i_r & \qquad I_r \ \mathrm{is \ an \ internal \ vertex \ of} \ \mathcal{G} \\ \\
|I_r|-1-i_r & \qquad I_r \ \mathrm{is \ an \ external \ vertex \ of} \ \mathcal{G}.  \\
\end{array} \right. $$

We define $v^*=a(v^*)b(v^*) \cdot D(v^*)$, where  $$a(v^*)=\prod_{i \in T}a_i,  \qquad b(v^*)=b(v),\qquad D(v^*)=\prod_{r=1}^m D_{I_r}^{j_r}.$$
\end{defn}

\begin{rem} \label{*}
It follows from Definition \ref{standard} that the dual of a standard monomial $v$ is a well-defined standard monomial. Furthermore, one has the property $v^{**}=v$, which shows that * is an involution. 
This gives an explicit duality between the standard monomials in complementary degrees. 
\end{rem}

\begin{ex} Let $v$ be the monomial in Example \ref{ex}. It is easy to verify that $v$ is a standard monomial. The dual of $v$ is equal to $a_1a_9 \prod_{i=1}^7 D_{I_r}^{i_r}$, where $i_1=i_2=i_4=1$ and $i_3=i_5=i_6=i_7=2$.
\end{ex}

\subsection{The pairing $R^d(X[n]) \times R^{n-d}(X[n])$}
In Definition \ref{dual} we defined dual elements for standard monomials. Below, we will see that the resulting intersection matrix between the standard 
monomials and their duals consists of square blocks on the main diagonal, whose entries are, up to a sign, intersection numbers in $R^{|S|}(X^{S})$, for certain sets $S$, and all blocks under the diagonal are zero. 
To prove the stated properties of the intersection matrix, we introduce a natural filtration on the tautological ring.

\begin{defn}\label{filter}
Let $v$ be a standard monomial as given in Definition \ref{standard}, and let $J_1, \dots ,J_s$ be the roots of the associated forest. Define $p(v)$ to be the degree of the element $$a(v)b(v) \cap_{r=1}^s X_{J_r} \in A^*(X^n),$$ 
which is the same as the integer $$\deg{a(v)b(v)}+\sum_{r=1}^s|J_r|-s.$$ The subspace $F^p R^*(X[n])$ of the tautological ring is defined to be the $\mathbb{Q}$-vector space generated by standard monomials $v$ satisfying $p(v) \geq p$.
\end{defn}

\begin{prop}\label{fil}
(a) For any integer $p$, we have that $F^{p+1}R^*(X[n]) \subseteq F^p R^*(X[n])$.

(b) Let $v \in F^pR^*(X[n])$ and $w \in R^d(X[n])$ be such that $w \ll v$. If $p+d > n$, then $v \cdot w$ is zero. In particular, $F^{n+1} R^*(X[n])$ is zero.
\end{prop}
\begin{proof}
The proof is similar to the proof of Proposition 5.15 in [T1].
\end{proof}

Using property (b) in the previous proposition we can show the triangular property of the intersection matrix:

\begin{prop}\label{tri}
Suppose $v_1,v_2 \in R^d(X[n])$ are standard monomials satisfying $D(v_1)<D(v_2)$. Then $v_1 \cdot v_2^*=0$.
\end{prop}
\begin{proof}
The monomial $v_1 \cdot v_2^*$ can be written as a product $v \cdot w$, for $v,w \in R^*(X[n])$ satisfying the properties given in Proposition \ref{fil}. For the argument see Proposition 5.16 in [T1].
\end{proof}

\begin{thm}\label{Ful}
For any $0 \leq d \leq n$, the pairing $$R^d(X[n]) \times R^{n-d}(X[n]) \rightarrow \mathbb{Q}$$ is perfect. In other words, the tautological ring $R^*(X[n])$ is a Gorenstein algebra with socle in degree $n$.
\end{thm}
\begin{proof}
As we have observed in Remark \ref{st}, the tautological group $R^d(X[n])$ is generated by standard monomials. 
By Remark \ref{*}, the involution * gives a one-to-one correspondence between standard monomials in degrees $d$ and $n-d$. 
We will prove that the resulting intersection matrix $(v_i \cdot v_j^*)$ is invertible, where the $v_i,v_j$ vary over the set of standard monomials in degree $d$. 
From the proven result in Proposition \ref{tri}, it is enough to show the invertibility of the intersection matrix consisting of all such $v$'s having the same $D$-part.
Let $$D \in \mathbb{Q}[D_I: I \subset \{1, \etc , n \} \ \mathrm{and} \ |I| \geq 3]$$ be a monomial. 
Denote by $\mathcal{G}$ the graph associated to the monomial $D$ and define the subset $S$ of the set $\{1, \etc , n\}$ as in Definition \ref{standard}. 
Now suppose that $v_i,v_j$ satisfy $D(v_i)=D(v_j)=D$.
From the study of the intersection ring of $X[n]$, it is easy to see that the intersection numbers
$$v_i \cdot v_j^* \in R^n(X[n]) = \mathbb{Q}$$ and the number $$a(v_i)b(v_i) \cdot a(v_j^*) b(v_j^*) \in R^{|S|}(X^S) = \mathbb{Q}$$
differ by $(-1)^{\varepsilon}$, where $\varepsilon=| \cup_{r=1}^m I_r| + \sum_{i \in V(\mathcal{G})} \deg(i)$. 
This shows that the intersection matrices $(v_i \cdot v_j^*)$ and $(a(v_i) b(v_i) \cdot a(v_j^*) b(v_j^*))$, for standard monomials having the same $D$-part $D$, are the same up to a sign. 
According to the study of the tautological ring of $X^S$ in the first section, we know that the kernel of the later matrix is generated by relations in $R^*(X^S)$.
After choosing a basis for $R^{d- \deg(D)}(X^S)$, the resulting intersection matrix is invertible. This proves the claim. 
\end{proof}

\begin{rem} 
In the proof of the theorem we used the fact that the tautological group in top degree is isomorphic to $\mathbb{Q}$. It is indeed easy to see that every element of degree $n$
in $R^*(X[n])$ is a multiple of the class $\prod_{i=1}^n a_i$. 
\end{rem}

\section{The tautological ring $R^*(M_{2,n}^{rt})$}
In this section we will prove that the tautological ring of $M_{2,n}^{rt}$ is isomorphic to the ring $R^*(X[n])$ studied in the previous section. 
As we have seen in the second section, there is a morphism $$F:X[n] \rightarrow M_{2,n}^{rt},$$ induced from the family of curves $\pi:X[n+1] \rightarrow X[n]$. The morphism $F$ induces the pull-back homomorphism 
$$F^*: A^*(M_{2,n}^{rt}) \rightarrow A^*(X[n]).$$ We want to see that the homomorphism $F^*$ sends the tautological classes to the tautological classes, and hence, it induces a ring homomorphism between the tautological rings involved. 
Notice that the tautological ring of $M_{2,n}^{rt}$ is generated by the divisor classes $\psi_i$ and $D_I$. Let us recall the definition of the divisor $D_I$ for a subset $I \subset \{1, \etc,n\}$ 
with $|I| \geq 2$: The generic curve parameterized by $D_I$ has two components. One component is of genus two and the other component is a rational curve. The elements of $I$ correspond to the markings on the rational component. 
It is easy to see that $$ F^*(D_{i,j})=d_{i,j}-\sum_{i,j \in I}D_I,  \qquad F^*(D_I)=D_{I} \ \mathrm{where} \ |I| \geq 3.$$
We denote the class $d_{i,j}-\sum_{i,j \in I}D_I$ in $R^1(X[n])$ by the letter $D_{i,j}$ so that the equality $F^*(D_I)=D_I$ holds for every subset $I$ of the set $\{1, \etc , n\}$ with at least two elements. 
Using this notation, the pull-back of the class $\psi_i$ via $F$ is computed as $K_i+\sum_{i \in I}D_I$. 
We conclude that the pull-backs of the tautological classes on $M_{2,n}^{rt}$ along the morphism $F$ belong to the tautological ring of the space $X[n]$.
This observation also shows that $F^*$ is a surjection onto the tautological algebra $R^*(X[n])$. 

The {\em injectivity} of $F^*$ is proven by verifying that the tautological classes $\psi_i$ and $D_I$ on $M_{2,n}^{rt}$ satisfy all the relations which hold among their images via $F^*$. 

\begin{notation}
We define the tautological classes $a_i,b_{j,k}$ on the moduli space $M_{2,n}^{rt}$ as follows: First, denote by $K_i$ the class $\psi_i-\sum_{i \in I} D_I$ and by $d_{j,k}$ the class $\sum_{j,k \in I}D_I$. 
The classes $a_i,b_{j,k}$ are defined via the equations $K_i=2a_i$ and $b_{j,k}=d_{j,k}-a_j-a_k$.
\end{notation}

The verification of the relations on the moduli space is done in several steps:

\begin{itemize}
\item
We first consider the relations among the generators of $R^*(X^n)$: The first relation is $a_1^2=0$. The desired relation on $M_{2,1}^{rt}$ follows from the study of the Chow ring of $\overline{M}_{2,1}$ by Faber. In [F1] Faber finds a relation in $A^2(\overline{M}_{2,1})$, which expresses $\psi_1^2$ in terms of the boundary cycles on the complement of $M_{2,1}^{rt}$. This means that $\psi_1^2=0$ as an element in $A^2(M_{2,1}^{rt})$. 
The result follows since on $M_{2,1}^{rt}$ the class $\psi_1$ is $K_1=2a_1$. 

Then we look at the relations in $R^*(X^2)$: In [G] a topological recursion formula in genus two is proven. That relation gives an explicit formula for $\psi_1 \psi_2$ in terms of the boundary cycles in $R^2(\overline{M}_{2,2})$. The restriction of that relation
to $M_{2,2}^{rt}$ gives $\psi_1 \psi_2 +3D_{1,2}^2=0$. Using the fact that the products $\psi_1 \cdot D_{1,2}, \psi_2 \cdot D_{1,2}$ are zero, we get that $d_{1,2}^2=-2a_1a_2$. This equation together with $\psi_1 \cdot D_{1,2}=0$ gives the relation 
$a_1 b_{1,2}=0$. The equality $b_{1,2}^2=-4a_1a_2$ follows from these relations. 

There are two ways to observe the relation $b_{1,2}b_{1,3}=a_1b_{2,3}$ on the moduli space: 
The pull-back of the relation \eqref{F} to $R^2(M_{2,3}^{rt})$ via the morphism $M_{2,3}^{rt} \rightarrow C_2^3$ is twice 
$b_{1,2}b_{1,3}-a_1b_{2,3}=0$.
The second proof follows from the the relation found by Belorousski-Pandharipande in $R^2(\overline{M}_{2,3})$:
The restriction of the relation in [BP] to the space $M_{2,3}^{rt}$ is 12 times $a_1b_{2,3}-b_{1,2}b_{1,3}=0$.
The next case is treated similarly using the relation we proved in $R^3(C_2^6)$: We just need to consider the pull-back of the relation \eqref{6} to $R^3(M_{2,6}^{rt})$ via the morphism $M_{2,6}^{rt} \rightarrow C_2^6$.

\item We now consider the first class of relations among the divisor classes $D_I$. It is easy to see that the product $D_I \cdot D_J \in R^2(M_{2,n}^{rt})$ is zero unless 
$$(**) \qquad I \subseteq J, \mathrm{or} \qquad  J \subseteq I, \mathrm{or} \qquad I \cap J = \emptyset.$$
Notice that this holds even in $R^*(\overline{M}_{2,n})$ for all subsets $I,J$ with at least two elements. In $(*)$ we only consider the case of subsets with at least three elements. 
 
\item
Let $I \subset \{1,\etc,n\}$ be a subset with $|I| \geq 3$ and $i_I:X_I \rightarrow X^n$ denote the inclusion. The $\ker(i_I^*)$ is generated by the divisor classes $d_{j,k}+2a_j$ and $d_{i,j}-d_{i,k}$, for $i \in I^c$ and distinct elements $j,k \in I$.  
We will prove that $( \psi_j - \sum_{j \in J, k \notin J} D_J) \cdot D_I$ and $(\sum_{i,j \in J} D_J - \sum_{i,k \in J} D_J) \cdot D_I$ are zero in $R^2(M_{2,n}^{rt})$. 
The first equality follows from the well-known formula for the $\psi$ classes in genus zero, which we recall: 
Let $i \in \{1, \etc , n\}$ be an element and assume that $j,k \in \{1, \etc, n\} \backslash \{i\}$ are arbitrary distinct elements. Then one has the following equality in 
$A^1(\overline{M}_{0,n})$: $$\psi_i=\sum_{\substack{i \in I \\j,k \notin I}}D_I.$$ 
The second relation is an easy implication of the relations $(**)$ in the previous paragraph.  

\item Let $V=V_1 \cap \dots \cap V_k,W$ and $Z$ be subvarieties of $X^n$ as in \ref{R}, so that $V \cap W=Z$. After possibly relabeling the indices, we can assume that 
$$Z=X_{I_0}, \qquad V_i=X_{I_i}, \ \mathrm{for} \ 1 \leq i \leq k, \qquad W=\prod_{i=2}^{r_1} d_{1,i} \cdot \prod_{j=1}^k d_{1,r_j+1} ,$$
where $1 \leq r_1 < \dots < r_{k+1} \leq n, I_0=\{1, \dots ,r_{k+1}\}$, and $I_i=\{r_i+1, \dots , r_{i+1}\}$ for $1 \leq i \leq k$. 
From these data we get the relation 
$$P_{W/X^n}(-\sum_{I_0 \subseteq I} D_I) \cdot \prod_{i=1}^k D_{I_i}=0$$ on the space $X[n]$,
for $$P_{W/X^n}(t)=\prod_{i=2}^{r_1} (t+d_{1,i}) \cdot \prod_{j=1}^k (t+d_{1,r_j+1}).$$ 
We want to prove a similar identity on the moduli space $M_{2,n}^{rt}$.
This is proven by showing that any monomial in the expansion of this expression is zero. The proof is easy and can be found in [T1].

\item Let $I \subset \{1,\etc,n\}$ be a subset with $|I| \geq 3$ corresponding to the subvariety $Z=X_I$ of $X^n$. The relation $P_{Z/X^n}(-\sum_{I \subseteq J}D_J)=0$ on the space $X[n]$ holds. 
To prove the corresponding relation on the moduli space we will show that the product $$\prod_{i \neq j \in I} (d_{i,j}- \sum_{I \subseteq J} D_J)$$ is zero, 
where $i \in I$ is an arbitrary element.
This is verified as in the previous case by showing that the monomials occurring in the expansion of the expression above are all zero. 
\end{itemize}

The argument above proves that $F^*$ is indeed an isomorphism. This gives the description of the tautological ring $R^*(M_{2,n}^{rt})$ in terms of generators and relations. 
In particular, from the proven result in Theorem \ref{Ful}, we finish the proof of Theorem \ref{rt}.

\end{document}